\documentclass[10pt,a4paper]{article}

\usepackage{epsfig}
\usepackage{amssymb}
\usepackage{latexsym}

\usepackage[ulem=normalem]{changes}

\newtheorem{theorem}{Theorem}[section]
\newtheorem{lemma}[theorem]{Lemma}
\newtheorem{proposition}[theorem]{Proposition}
\newtheorem{definition}[theorem]{Definition} \newtheorem{corollary}[theorem]{Corollary}
\newtheorem{remark}[theorem]{Remark} 
\newenvironment{proof}[1][Proof]{\removelastskip\vskip12pt
plus 1pt \noindent\em {\bf #1. }\rm}{\hfill$\Box$ \vskip.2cm}

\newcommand{\R}{\mathbb{R}}
\renewcommand{\d}{\,\mathrm{d}}
\usepackage{amsmath}
\usepackage{amsfonts}
\usepackage{amssymb}
\usepackage{color}
\definecolor{mauve}{rgb}{0.5,0,0.5}

\renewcommand{\leq}{\leqslant}
\renewcommand{\geq}{\geqslant}
\renewcommand{\le}{\leqslant}
\renewcommand{\ge}{\geqslant}
\renewcommand{\L}{\mathrm{L}}
\newcommand{\C}{\mathrm{C}}
\newcommand{\e}{\,\mathrm{e}\,}
\newcommand{\loc}{\mathrm{loc}}
\newcommand{\supp}{\mathop{\rm supp}}

\definechangesauthor{EC}{mauve}
\definechangesauthor{CB}{blue}
\definechangesauthor{RF}{red}
\setauthormarkup[]{}

\begin{document}

\begin{center}
{\Large \textbf{Unbounded solutions of the nonlocal heat equation}}\\
\vspace*{0.5in}
{\large
\sc Cristina Br\"andle\footnote{\noindent  Departamento de Matem\'{a}ticas, U.~Carlos III de Madrid,
28911 Legan\'{e}s, Spain\\e-mail: {\tt cristina.brandle@uc3m.es}} \& Emmanuel Chasseigne\footnote{Laboratoire de Math\'{e}matiques et Physique Th\'{e}orique, U.~F. Rabelais, Parc de Grandmont, 37200 Tours, France\\
email: {\tt emmanuel.chasseigne@lmpt.univ-tours.fr}}
\& Ra\'ul Ferreira\footnote{Departamento de , U.~Complutense de Madrid,
280xx Madrid, Spain\\
email: {\tt raul$\underline{\ }$ferreira@mat.ucm.es}}
}
\end{center}
\vspace*{0.5in}

\begin{abstract}
\noindent
We consider the Cauchy problem posed in the whole space for the
following nonlocal heat equation: $ u_t = J\ast u -u\,,$
where $J$ is  a symmetric continuous probability density. Depending on the tail of $J$,
we give a rather complete picture of the problem in optimal classes of data by:
$(i)$ estimating the initial trace of (possibly unbounded) solutions;
$(ii)$ showing existence and uniqueness results in a suitable class;
$(iii)$ giving explicit unbounded polynomial solutions.
\end{abstract}

\



\

\section{Introduction}
\label{sect:intro}

This paper is concerned with the Cauchy problem in optimal classes of data for the
following nonlocal version of the heat equation:
\begin{equation}\label{eq:0}
 u_t = J\ast u -u,\quad(x,t)\in\R^N\times\R_+\,.
\end{equation}
Here, $J:\mathbb{R}^N\to\mathbb{R}$ is  a symmetric continuous probability density
and $f\ast g$ stands for the convolution of functions $f$ and $g$. Throughout the paper, the initial data
\begin{equation}
  \label{eq:initial.data}
  u(x,0)=u_0(x),\quad x\in\R^N,
\end{equation}
are only assumed to be locally bounded (and not necessarily nonnegative).

This equation and some of its variants have been studied by a number of authors recently
and in various directions (see for instance \cite{ChasseigneChavesRossi07} and the references therein). These works are essentially dealing with the class of either bounded or $\L^1$ initial data.
On the contrary, the aim of this paper is to study the Cauchy problem in classes of data that are not
necessarily bounded, nor integrable, so that the ``usual'' tools (Fourier transform, fixed-point theorems)
do not work.

Let us mention first that in the case of continuous and compactly supported initial data $u_0$,
the solution of \eqref{eq:0}--\eqref{eq:initial.data} can be written as:
$$u(x,t)=\e^{-t}u_0(x)+(\omega(t)\ast u_0)(x)\,,$$
where $\omega$ is the regular part of the nonlocal heat kernel (see \cite{ChasseigneChavesRossi07} and an expansion
of $\omega$ in Section \ref{sect:def.trace}).

Now, as is well-known, in the case of the ``classical'' heat equation $u_t=\Delta u$ in $\R^N\times\R_+$,
the optimal class consists of initial data of functions $u_0$ satisfying the estimate:
\begin{equation}\label{est:local.heat}
    |u_0(x)|\leq c_0 \e^{\alpha |x|^2}\,,\quad \text{for some }\alpha,c_0>0\,.
\end{equation}
Uniqueness also holds in the class of solutions satisfying this type of estimate -- see for instance \cite{JohnBook}.
In this paper, we shall prove that nonnegative weak solutions have an initial trace which is a locally integrable function (see more precise formulations
in Propositions \ref{prop:trace}).

We would like to mention that after this paper was finished in its preprint version,
we came across a recent article by Natha\"el Alibeau and Cyril Imbert \cite{AI}, so that both
articles were written independently.
In \cite{AI}, the authors also deal with unbounded solutions of a more general equation.
They also consider other interesting aspects like regularization by the equation, that we do not
investigate here. On the contrary, they essentially treat the case of the Fractional Laplacian
while we are interested here in treating more general situations than power-type kernels.

In order to explain more in detail the main differences between \cite{AI} and the present article, let us
make several important remarks:
\begin{enumerate}
	\item Let us first mention that in this paper, we could treat singular Lévy measures $\mu$ in the equation:
	$$u_t(x,t)=\int_{\R^N}\big\{u(x+z,t)-u(x,t)\big\}\d\mu(z)\,.$$
	Indeed, the question of obtaining the optimal behaviour of initial data is only linked to the tail of $\mu$
	and not to the presence of singularities at the origin. Following the notation in \cite{AI},
	we can always decompose the Lévy measure $\mu$ as the sum of its singular part near the origin,
	and the bounded measure $\mu_b=\mu\mathbf{1}_{\{|x|>1\}}$ which conveys the
	tail information. In this paper, we only deal with the $\mu_b$-part since the singular part does not play any
	role here.	We refer to Section \ref{subsect:singular} for some discussion on that issue.
	
    \item  We would like to put emphasis on the fact that in the problem we are facing, one can divide the
	measures $\mu$ (or kernels $J$) into two categories, according to the speed of decay at infinity.
	Following \cite{BrandleChasseigne09-1,BrandleChasseigne09-2}, we say that the kernel has a \textit{slow decay}
	if $J(y)=\e^{-|y|\omega(y)}$ with $\omega(y)\to0$ as $|y|\to\infty$. On the contrary, the kernel has a
	\textit{fast decay} if $\omega(y)\to\infty$. The limit case of exponential decay is somewhat in between,
	but here it may be considered as being ``slow decay''-type.
	
    \item In the case of \textit{slow decay} kernels, using elliptic barriers
	works quite well: if one can find a function $f\geq0$ and a $\lambda>0$ such that
	$J\ast f -f\leq \lambda f$, then the function $$\psi(x,t):=\e^{\lambda t}(1+f(x))$$
	is a natural supersolution of the problem. This is essentially what is done in \cite{AI} in the case $J(y)=|y|^{-N-\alpha}$,
	for which $f(x)=|x|^{\beta}$ for any $0<\beta<\alpha$ gives rise to a supersolution. Hence the problem is at least solvable
	in the class of initial data such that $|u_0(x)|\leq c_0(1+f(x))$. We construct here other explicit barriers for
	different kernels, for instance the case of \textit{tempered $\alpha$-stable laws} (and CGMY-processes, see \cite{CarrGemanMadanYor2003})
	for which $J(y)=\e^{-|z|}/|z|^{N+\alpha}$, see Theorem~\ref{thm:3}.	
	
    \item A natural related question is to prove that the barriers actually give the optimal
	behaviour for the initial data. Although this is not stated in \cite{AI} as such, the existence result obtained there
	is optimal for power-type kernels (see Section \ref{subsect:power} below). In all the slow decaying kernels we consider,
	we prove that the explicit elliptic barriers we construct are optimal.
	
    \item On the contrary, in the case of \textit{fast decaying} kernels the elliptic barriers \textbf{do not} give the optimal
	class of initial data, which does not allow to treat those cases in full generality.
	Instead, one has to use another approach by estimating directly the non local heat kernel in a different way
	and construct other supersolutions than the ones in \cite{AI}, more adapted to the problem. We consider, the typical cases of gaussian laws and compactly supported kernels.
	
    \item We also investigate related problems like the phenomenon of blow-up in finite time, asymptotic behaviour
	of unbounded solutions and we construct explicit polynomial solutions.
\end{enumerate}

Here are now our main results.

\noindent\textsc{Fast decaying kernels --} Let us take a symmetric kernel $J$ such that the support of $J$
is exactly $B_1=\{|x|\leq1\}$. It is known that in this case, equation \eqref{eq:0} shares a lot of properties with the
local heat equation, \cite{ChasseigneChavesRossi07}. So, is the class of optimal data for~\eqref{eq:0} also given by \eqref{est:local.heat}? The answer
is no: instead of the quadratic exponential growth we get a $|x|\ln|x|$-exponential growth, but the result
is similar in its form:
\begin{theorem}\label{thm:1}
    Let $J$ be a radially symmetric kernel, supported in $B_1$, let $u_0$ be
    a locally bounded initial data and $c_0>0$ arbitrary. Then the following holds:\\[6pt]
    $(i)$ If $|u_0(x)|\leq c_0\e^{\alpha|x|\ln|x|}$ for some $0<\alpha<1$, then there exists a global solution of \eqref{eq:0}--\,\eqref{eq:initial.data}.\\[6pt]
    $(ii)$ If $u_0(x)\geq c_0\e^{\beta |x|\ln|x|}$ for some $\beta>1$, then there exists no nonnegative solution of \eqref{eq:0}--\,\eqref{eq:initial.data}.\\[6pt]
    $(iii)$ For any $0<\alpha<1$, there exists at most one solution such that $|u(x,t)|\leq C(t)\e^{\alpha|x|\ln|x|}$,
    where $C(\cdot)$ is locally bounded in $[0,\infty)$.
\end{theorem}
Let us mention that this theorem cannot be obtained by using elliptic barriers, since they do not give the optimal
$(|x|\ln|x|)$-behaviour. At least, for any $\alpha>0$, the function $\e^{\alpha |x|}$ is an elliptic barrier which
is enough to get uniqueness results in the class of solutions and initial data that grow at most like exponentials.
On the other hand, trying a barrier like $f(x)=\e^{\alpha|x|\ln|x|}$ leads at best to the estimate
$J\ast f-f\leq \lambda |x|^\alpha f$, which does not allow to construct a suitable supersolution $\psi$ as above.

Thus, this theorem comes from a direct estimate of the ``heat kernel'' $\omega(x,t)$, using its semi-explicit
series expansion \eqref{eq:def.omega}, see Section~\ref{sect:optimal.II}.

\noindent\textsc{Blow-up in finite time --} In the case of the usual heat equation, it is well-known that if the
initial data has a critical growth of the order of a quadratic exponential, then the corresponding solution blows up in finite time.
We prove here that it is also the case for gaussian kernels:
\begin{theorem}\label{thm:2}Let $J$ be a centered gaussian with variance $\sigma^2=1/2$, let $u_0$ be
    a locally bounded initial data and $c_0>0$ arbitrary such that
	$$u_0(x)=c_0\e^{|x|(\ln|x|)^{1/2}+f(|x|)}\quad for\ some\ \alpha<0<\beta,\ \alpha s\leq f(s)\leq \beta s\,.$$
    then the minimal solution of~\eqref{eq:0}--\eqref{eq:initial.data} blows up in finite time.
\end{theorem}
For more complete results on existence, nonexistence and uniqueness, see Theorem~\ref{thm:gaussian.ex}.
We conjecture that Theorem \ref{thm:2} is also valid for compactly supported kernels with a $(|x|\ln|x|)$-critical rate
instead (and more general kernels, if not all of them), we hope to address this question in a future work.

\noindent\textsc{Optimal trace for slow decaying kernels --} We prove that the estimates given by the
elliptic barriers are optimal in a number of slow decaying cases.
Let us just illustrate this with a result concerning \textit{tempered $\alpha$-stable law}-type kernels:
\begin{theorem}\label{thm:3}
    Let $J$ be a kernel satisfying
    $$\alpha_0=\sup\Big\{\alpha>0: \int J(y)\e^{|y|}(1+|y|)^{N+\alpha}\d y< \infty\Big\}<\infty\,.$$
    Let $u_0$ be a locally bounded initial data and $c_0>0$. Then the following holds:\\[6pt]
    $(i)$ If  $|u_0(x)|\le c_0\e^{|x|}(1+|x|)^{N+\alpha}$ for $0<\alpha<\alpha_0$, then there exists a  global solution of \eqref{eq:0}--\,\eqref{eq:initial.data}.\\[6pt]
    $(ii)$  If $u_0(x)\ge c_0\e^{|x|}(1+|x|)^{N+\alpha}$ for $\alpha\ge \alpha_0$, then there exists no nonnegative solution of   \eqref{eq:0}--\,\eqref{eq:initial.data}.\\[6pt]
    $(iii)$ For any $0<\alpha<\alpha_0$ there exists at most one solution such that
    $$
    |u(x,t)|\ll c_0\e^{|x|}(1+|x|)^{N+\alpha} \mbox{ locally uniformly in }\ [0,\infty).
    $$
\end{theorem}
This theorem applies for $J(y)=\e^{-|y|}/(1+|y|)^{N+\alpha}$, but also with some adaptations
to the singular case $J(y)=\e^{-|y|}/|y|^{N+\alpha}$, as we mentioned at the beginning of the introduction (see also
Section~\ref{subsect:singular}).

\noindent\textsc{Asymptotic behaviour --} It is well-known that if the initial data is integrable and bounded, then
$u(t)$ converges to zero as $t\to\infty$ (see \cite{ChasseigneChavesRossi07}). But if we consider solutions
given by an initial data such that $u_0(x)\to+\infty$ as $|x|\to\infty$, what can we get as asymptotic behaviour?
We prove here that there exist explicit polynomial solutions which are similar to the ones for the local heat equation.
Moreover, these solutions tend to $+\infty$ as $t\to\infty$. For instance:

\begin{theorem}\label{thm:4}
    Let $J$ be a kernel with finite second-order momentum $m_2>0$. Then
    if $u_0=|x|^2$, the solution  of~\eqref{eq:0}--\,\eqref{eq:initial.data} has the explicit form
    \[
    u(x,t)=|x|^2+m_2t.
    \]
    More generally, if $J$ has a momentum of order $2p$ and  $u_0=|x|^{2p}$, then
    $$u(x,t)=|x|^{2p}+m_{2p}t^p+o(t^p)\quad as\quad t\to\infty\,.$$
\end{theorem}

To end this introduction, let us mention that more general  nonexistence result for changing sign solutions can be proved provided the negative
(or positive) parts of the initial data and the solution are controlled, see Remark~\ref{rem.nonex}.
Also, similar results are easily obtained for only locally integrable data,	provided they satisfy the various estimates for
$|x|$ large. Indeed, this is just a matter of decomposing $u_0$ as the sum of a locally bounded function and an integrable function for which we can solve the problem	separately, using the linearity of the equation.

Finally, let us mention that the bounds we obtain are related to the large deviations estimates
given in \cite{BrandleChasseigne09-1,BrandleChasseigne09-2}, because in both cases we are somehow estimating the
nonlocal heat kernel of the equation, see also Section~\ref{sect:extension}.

\textbf{Organization --}
Section~\ref{sect:def.trace} shows that the initial trace exists for nonnegative solutions
and give some basic regularity results.
Then we show in Section~\ref{sect:ex.un} general existence and uniqueness results provided there exists
a positive supersolution of the problem.
Section \ref{sect:optimal} is devoted to apply this to various examples of slow decaying
kernels, and prove that we obtain optimal existence and uniqueness results. In Section~\ref{sect:optimal.II} we deal with fast decaying kernels.
In Section \ref{sect:explicit} we construct explicit polynomial solutions
which tend to $+\infty$ as $t\to\infty$ if the initial data is unbounded.
Finally we end with a section of comments, further results
and possible extensions.

\section{Preliminaries}
\label{sect:def.trace}
\setcounter{equation}{0}

First of all let us fix some notation that will be used in what follows:

\begin{itemize}
  \item[$(i)$] We denote by $B_r$ be the ball of radius $r$ centered at $0$ and by $\chi_r$ its characteristic function.
  \item[$(ii)$] Throughout the paper, $J:\R^N\to\R$ is a nonnegative symmetric continuous function such that $\int_{\R^N} J(y)\d y=1$;
the notation $J^{\ast n}$ denotes the convolution of $J$ with itself $(n-1)$-times, $n$ counting the number of $J$'s, so that
by convention $J^{\ast 0}=\delta_0\,,\ J^{\ast 1}=J$.
    \item[$(iii)$] If $\tau:\R^N\to\R$ is a smooth, nonnegative and compactly supported in $B_1$ function, such that $\int\tau=1$,
then $\tau_n$ stands for the resolution of the identity given by $\tau_n(x)=n^N\tau(nx)$.
    \item[$(iv)$] If $f,g:\R^N\to\R$ we write $f\ll g$ if $f(x)=g(x)\epsilon(x)$ with $\epsilon(x)\to0$ as $|x|\to\infty$. But moreover
we extend it as follows for time-depending functions:
\end{itemize}

\begin{definition}\label{def:neg.loc.uniform}
    If $u,v:\R^N\times[0,\infty)\to\R$, we say that $u\ll v$ locally uniformly in $[0,\infty)$ if
    $$|u(x,t)|\leq C(t)\epsilon(x)|v(x,t)|\,,$$
    for some $C\in\L^\infty_{\rm loc}([0,\infty))$ and $\epsilon(x)\to0$ as $|x|\to\infty$\,.
\end{definition}
This definition will be useful in the construction of a supersolution $\psi$ which will tend to $+\infty$ as $t\to\infty$. Notice that definition does not require  to have a uniform control for all $t\geq0$.

\subsection{The nonlocal heat kernel}

It is shown in~\cite{ChasseigneChavesRossi07} that
if both the initial data and its Fourier transform are integrable,
then the solution to~\eqref{eq:0}--\eqref{eq:initial.data} is unique and given by:
\begin{equation}
  \label{eq:u.chacharo}
  u(x,t)=\e^{-t}u_0(x)+\big(\omega(t)\ast u_0\big)(x)\,,
\end{equation}
where $\omega$ is smooth. Of course, this is true for instance if the initial data is in $\L^2(\R^N)$ or bounded and compactly supported
(we shall use this result in various constructions in the sequel). The following lemma collects some basic properties of the regular part
of the nonlocal heat kernel:

\begin{lemma}\label{lem:omega}
	The function $\omega$ is given by
	\begin{equation}\label{eq:def.omega}
    	\omega(x,t)=\e^{-t}\sum_{n=1}^\infty\frac{t^nJ^{\ast n}(x)}{n!}\,.
	\end{equation}
	Moreover, $\omega$ is a solution of the following problem:
	\begin{equation}\label{eq:omega}
	\begin{cases}
	\omega_t=J\ast\omega-\omega+\e^{-t}J\,,\\
	\omega(x,0)=0\,.
	\end{cases}
	\end{equation}
\end{lemma}
\begin{proof}
	If $u_0$ and its Fourier transform are integrable, we can write the solution in frequency variables as follows:
	$$
	 \hat{u}(\xi,t)=\e^{-t}\Big(1+\sum_{n=1}^\infty\frac{t^n}{n!}\hat{J}^n(\xi)\Big)\hat u_0(\xi)\,.
	$$
	Hence, going back to the original variables we get
	$$
    u(x,t)= \e^{-t}\Big(\delta_0(x)+\sum_{n=1}^\infty\frac{t^n}{n!}J^{*n}(x)\Big)*u_0(x)
    =\e^{-t}u_0(x)+\big(\omega(t)\ast u_0\big)(x)\,,
	$$
	where $\omega$ is given by~\eqref{eq:def.omega}. Notice that since $\|J^{\ast n}\|_1=\|J\|_1=1$\,, then the series converges in $\L^1$.
	Moreover, a direct computation shows that $\omega(x,0)=0$ (in the continuous sense) and
	$$\begin{aligned}
    \omega_t & =-\omega+\e^{-t}\Big(J+\sum_{n=2}^\infty\frac{nt^{n-1}J^{*n}}{n!} \Big)=
	-\omega+\e^{-t}\Big(J+\sum_{n=1}^\infty\frac{t^{n}J^{*(n+1)}}{n!} \Big)\\
	& =-\omega+\e^{-t}\Big(J+J*\sum_{n=1}^\infty\frac{t^{n}J^{*n}}{n!} \Big)=J\ast\omega -\omega+\e^{-t}J\,.
	\end{aligned}$$
	This ends the lemma.
\end{proof}

\subsection{Weak and Strong Solutions}

Let us now specify what are the notions of solution that we use throughout the paper:

\begin{definition}\label{def:weak}
    A weak solution of \eqref{eq:0} is a function $u\in\L^1_{\rm loc}(\R^N\times\R_+)$
    such that the equation holds in the sense of distributions.
\end{definition}

We may also, as usual, consider strong and classical solutions:
\begin{definition}\label{def:strong}
    Let $u_0\in \L^1_\loc(\R^N)$.

    \noindent (i) A strong solution of \eqref{eq:0} is a function $u\in \C^0\big([0,\infty);\L^1_{\rm loc}(\R^N)\big)$
    such that $u_t,J\ast u\in \L^1_\loc(\R^N\times\mathbb{R}_+\big)$, the equation
    is satisfied in the $\L^1_\loc$-sense and such that $u(x,0)=u_0(x)$
    almost everywhere in $\R^N$.

    \noindent (ii) A classical solution of \eqref{eq:0} is a solution such that moreover
    $u,u_t,J\ast u\in \C^0(\R^N\times[0,\infty))$ and the equation holds in the classical sense
    everywhere in $\R^N\times[0,\infty)$.

    \noindent (iii) A sub or supersolution is defined as usual with inequalities instead of equalities
    in the equation and for the initial data.

\end{definition}

The following result proves that in fact, nonnegative weak solutions are strong. This will allow us
to consider only strong solutions in the rest of the paper, even for changing sign solutions which will
be constructed as the difference between two nonnegative solutions. But let us also mention here
that actually, the same result will be valid for changing sign solutions provided they belong to the uniqueness
class, see Section \ref{sect:extension}.

\begin{proposition}\label{prop:trace}
    Let $u$ be a nonnegative weak solution of \eqref{eq:0}. Then $u$ has an initial trace $u(x,0^+)$ which
    is a nonnegative $\L^1_\loc$ function, and $u$ is a strong solution of \eqref{eq:0}.
    If moreover $u(x,0^+)$ is continuous, then $u$ is in fact a classical solution.
\end{proposition}
\begin{proof}
    We consider the auxiliary function $v(x,t):=\e^t u(x,t)$ which satisfies:
    $$ \partial_t v(x,t)=J\ast v\ge 0\,.$$
    This proves that the limit $v(x,0^+)$ is defined.
    Now, since $u\in\L^1_\loc(\R^N\times\R_+)$, then $u(\cdot,t)\in\L^1_\loc(\R^N)$ for almost
    any $t>0$, so that for such a time,
    $$0\leq u(x,0^+)=v(x,0^+)\leq v(x,t)= \e^{t}u(x,t)\in\L^1_\loc(\R^N)\,.$$
    This proves that the initial trace of $u$ is indeed a $\L^1_\loc$
    function. Even more, this also proves that  $u(\cdot,t)\in \L^1_\loc(\mathbb R^N)$ functions for all
    $t>0$. Now, we can write
\begin{equation}\label{eq:integrated.v}
    v(x,t)=v(x,0^+)+\int_0^t (J\ast v)(x,s)\d s\,,
\end{equation}
    which implies that $J\ast v$ and then $J\ast u$ are $\L^1_\loc(\R^N\times\R_+)$ functions.
    Since $u$ is already a $\L^1_\loc$ function, then $u_t$ is also a $\L^1_\loc$ function. Hence $u$ is a strong solution.
    Finally, if $u(x,0^+)=v(x,0^+)$ is continuous, then clearly by \eqref{eq:integrated.v},
    $v$ is continuous in space and time, up to $t=0$. Hence $u$ is continuous and the equation holds
    in the classical sense.
\end{proof}
The following technical trick will be useful in order to get some continuity for comparison results:

\begin{lemma}\label{lem:trick}
    Let $u$ be a strong solution to~\eqref{eq:0}--\eqref{eq:initial.data}. Then
    for any smooth and compactly supported function $\tau:\R^N\to\R$, the function
    $$u^{\tau}(x,t):=(\tau\ast u)(x,t)=\int_{\R^N}\tau(x-y)u(y,t)\d y$$
    is a classical solution of~\eqref{eq:0}--\eqref{eq:initial.data} with initial data
    $\tau\ast u_0$.
\end{lemma}
\begin{proof}
    Since $\tau$ is compactly supported and smooth, the convolutions $\tau\ast u_0$
    and $u^{\tau}$ are well-defined and moreover, elementary properties of the convolution show that
    $$\partial_t(u^{\tau})=\tau\ast u_t = \tau\ast(J\ast u) - \tau\ast u = J\ast u^{\tau} -u^{\tau}\,.$$
    Finally, $u^{\tau}$ is clearly continuous hence it is a classical solution
    and its initial trace is $\tau\ast u_0$.
\end{proof}
Of course, the same result is valid for sub/super solutions: since $\tau_n$ is nonnegative, the convolution
maintains the inequality.

\subsection{Supersolutions}

In order to establish a full theory of existence and uniqueness for solution of~\eqref{eq:0}--\eqref{eq:initial.data} we will use special supersolutions
$\psi$ that verify:
\begin{equation}\label{eq:def.psi}
    \begin{aligned}
    &\psi\in\mathrm{C}^0(\R^N\times[0,\infty))\,,\psi\geq0\,,\\\
    & \psi(x,t)\to+\infty\text{ as }|x|\to\infty\text{ uniformly for }t\in[0,\infty)\,,\\
    &\psi_t\geq J\ast\psi-\psi\,.
    \end{aligned}
\end{equation}

A typical way to construct such supersolutions is to use barriers, \text{i.e.,} functions
$f:\mathbb{R}^N\to \mathbb{R}$ continuous, nonnegative, with $f(x)\to\infty$ as $|x|\to\infty$ and such that for some $\lambda>0$,
\begin{equation}
  \label{eq:J.f}
  J*f-f\leq \lambda f, \qquad \lambda>0.
\end{equation}
Then $\psi(x,t):=\e^{\lambda t}f(x)$ satisfies \eqref{eq:def.psi}. For instance, if $J$ has a finite second-order momentum $m_2$,
then $f(x)=x^2$ fulfills the requirements, with $\lambda=m_2$, this is a calculus similar to the one performed in Section~\ref{sect:explicit}.

Another approach that we shall use in the case of fast decaying kernels is the following:
since $\omega$ satisfies~\eqref{eq:omega}, then it is a supersolution of equation \eqref{eq:0}. Thus, if $f$ is nonnegative and such that
the spatial convolution of $\omega$ and $f$ converges, then
the function $(x,t)\mapsto\big(\omega(t+1)\ast f\big)(x)$ is also a supersolution.

\section{Existence and uniqueness of unbounded solutions}
\label{sect:ex.un}
\setcounter{equation}{0}

\subsection{Comparison}

Comparison in the class of bounded solutions is well-known if $J$ is compactly supported. We extend  now the comparison result to more general kernels and not only for the class of bounded solutions. To do so, we need to control the points where  the maximum is attained, which is done by using the supersolution $\psi$ defined above.

\begin{proposition}\label{prop:comparison}
   	Let $\psi$ satisfy~\eqref{eq:def.psi}. Let  $\underline{u}$ be a strong subsolution of~\eqref{eq:0} and $\bar u$ a
	strong supersolution of~\eqref{eq:0} such that $\underline u(x,0)\leq \bar u(x,0)$. If $\underline u-\bar{u}\ll \psi$
	locally uniformly in $[0,\infty)$ then $\underline{u}\leq \bar u$ almost everywhere in $\R^N\times\R_+$.
\end{proposition}

\begin{proof}
    We shall do the proof in two steps: we assume first that the sub and supersolutions are continuous, and then use
    a regularization procedure to extend the result.

    \noindent\textsc{Step 1 --} Assume that both  $\underline{u}$ and $\bar u$ are continuous and
    consider the function
    \[
    w^\delta(x,t)=\e^t(\underline u(x,t)-\bar u(x,t)- \delta \psi(x,t)-\delta)\,,
    \]
    with $\delta>0$. This function satisfies
    $$
    \partial_t (w^\delta)-J*w^\delta\le 0.
    $$
    Let $(x_0,t_0)$ be the first point at which $w^\delta$ reaches the level $-\delta/2$ and assume that $t_0$ is finite.
	Since $\underline u-\bar u\ll \psi$ locally uniformly, we get that $x_0<\infty$.
    Moreover, since the function $w^\delta$ is continuous and $w^\delta(x,0)\le -\delta$, then $t_0>0$.

    Notice that $J*w^\delta\le -\delta/2$ for $t\in(0,t_0)$. Then, $\partial_t (w^\delta)\le 0$ for $t\in (0,t_0)$, which is a contradiction.
    Therefore, $w^\delta\le -\delta$ for all time $t>0$. Finally taking $\delta\to 0$, we obtain the desired result.

    \noindent\textsc{Step 2 --} Now we need to relax the continuity assumption.
    To this end, we use Lemma \ref{lem:trick} with a suitable resolution of the identity $\tau_n$ (see preliminaries)
	to see that $\underline{u}_n:=\tau_n\ast \underline{u}$ is a continuous subsolution and
    similarly, ${\bar u}_n:=\tau_n\ast{\bar u}$ is a continuous supersolution. Moreover, since $\tau_n\geq0$,
    their initial data are ordered:
    $$\underline{u}_n(x,0^+)=\tau_n\ast \underline u_0\leq \tau_n\ast \bar u_0={\bar u}_n(x,0^+)\,.$$
    In order to use Step 1 for the convolutions $\underline u_n$ and $\bar u_n$, let us notice first that if $\psi$ satisfies \eqref{eq:def.psi}, then
    $\psi_n:=\tau_n\ast\psi$ is also a supersolution (again, this is just convoluting the inequation). We only have to verify that
    $\underline{u}_n-\bar u_n\ll\psi_n$ locally uniformly.

    Actually, this comes from the fact that by definition, $\underline{u}-\bar u\ll\psi$ and since we take the convolution with a nonnegative $\tau_n$,
    $$\underline{u}_n-\bar{u}_n\leq c(t)\tau_n\ast(\epsilon(\cdot)\psi(\cdot,t))(x)\leq\epsilon'(x)(\tau_n\ast\psi)(x)\,,$$
    where $\epsilon'(x)=\max\{\epsilon(y): |y-x|\leq1/n\}$ still verifies $\epsilon'(x)\to0$ as $|x|\to\infty$.
    Hence, $\underline{u}_n-\bar{u}_n\ll\psi_n$, and we apply Step 1 with $\underline{u}_n, {\bar u}_n$ and $\psi_n$,
    which yields:
    $$\underline{u}_n\leq {\bar u}_n\quad\text{in}\quad\R^N\times[0,\infty)\,.$$
    Finally we pass to the limit as $n\to\infty$. Since $\underline{u}$ and $\bar u$ are locally integrable,
    we get $\underline{u}\leq\bar u$ almost everywhere.
\end{proof}

As direct corollary, we have the following uniqueness result,
which is valid only in a suitable class of solutions, as it is the case for the local heat equation.

\begin{theorem}
\label{thm:uniqueness}
    Let  $\psi$ satisfy~\eqref{eq:def.psi}.
    Then there exists at most one strong solution $u$ of~\eqref{eq:0}--\,\eqref{eq:initial.data}
    such that $|u|\ll\psi$ locally uniformly in $[0,\infty)$.
\end{theorem}
\begin{proof}
	If two solutions $u$ and $v$ satisfy $|u|,|v|\ll\psi$ locally uniformly in $[0,\infty)$ then
	we have at the same time $u-v\ll\psi$ and $v-u\ll\psi$ locally uniformly in $[0,\infty)$. Using
	these solutions as sub- and supersolutions allows us to obtain $u\leq v$ and $v\leq u$ so that
	they are identical.
\end{proof}

\subsection{Construction of a solution}

\begin{theorem}\label{thm:ex}
    Let $\psi$ satisfy~\eqref{eq:def.psi}.
    Let $u_0$ be locally bounded and assume that for some $c_0>0$,
    $$|u_0(x)|\leq c_0\psi(x,0)\quad\text{in }\R^N\,.$$ Then
    there exists a solution $u$ of~\eqref{eq:0}{\rm--}\eqref{eq:initial.data}, which is given by
	$$
	u(x,t)=\e^{-t}u_0 (x)+ (\omega(t)\ast u_0)(x).
	$$
    Moreover, the following estimate holds: $|u(x,t)|\leq c_0\psi(x,t)$ in $\R^N\times\R_+$.
\end{theorem}

\begin{proof}
   We separate the positive and negative parts of the initial data $(u_0)_+$ and $(u_0)_-$, and solve
   separately the two problems by linearity of equation~\eqref{eq:0}.
   Hence, let $u_n^+$ be the unique solution,
    see~\cite{ChasseigneChavesRossi07}, of the truncated problem
    \[
         \begin{array}{l@{\qquad}l}
          (u_n)_t=J*u_n-u_n, & (x,t)\in \mathbb{R}^n\times (0,\infty),\\[5pt]
          u_n(x,0)=(u_0)_+(x)\chi_n(x), & x\in\mathbb{R}^n.
        \end{array}
    \]
    Since, $u_n(x,0)\in\L^2(\R^N)$ the representation formula~\eqref{eq:u.chacharo} holds.
	Moreover, since $u_n(x,0)\in\L^\infty(\R^N)$, then
    \begin{equation}\label{e.aprox}
    u_n^+(x,t)=\e^{-t}u_n(x,0)+\big(\omega(t)\ast u_n(x,0)\big)(x)\le \|u_n(x,0)\|_\infty\,.
    \end{equation}
    So, we have obviously $u_n^+\ll\psi$ locally uniformly in $[0,\infty)$. Therefore, we can apply
    the comparison principle, Proposition~\ref{prop:comparison}, to get

    $(i)$ the sequence $\{u_n^+\}$ is monotone nondecreasing;

    $(ii)$ $u_n^+(x,t)\leq \bar c_0 \psi(x,t)$.

    Hence, there exists a limit
    $u^+$ defined in all $\mathbb{R}^N\times(0,\infty)$. We have to check that it is in fact a solution of~\eqref{eq:0}.
    Passage to the limit in the equation is done by using the dominated convergence for the convolution
    term: the limit function $u^+$ is in $\L^1_{\loc}(\mathbb{R}^N\times(0,\infty))$
    and $J\ast u_n^+$ converges to $J\ast u^+$ in $\L^1_\loc(\R^N)$, thus we recover a strong solution with initial data $(u_0)_+$\,.
	Moreover, passing to the limit in~\eqref{e.aprox} and in point $(ii)$ above we obtain
	$$
	u^+(x,t)=\e^{-t}(u_0)_+(x) + (\omega(t)\ast (u_0)_+)(x)\,,\qquad
	u^+(x,t)\leq \bar c_0\psi(x,t)\,.
	$$

    The same construction, using again $c_0\psi$ as a nonnegative supersolution gives a solution $u^-$
    with initial data $(u_0)_-$, the negative part of $u_0$.
    Thus, $u=u^+-u^-$ is a solution with initial data $u_0=(u_0)_+-(u_0)_-$.
    Comparison with $\psi$ is straightforward.
\end{proof}

\begin{remark}\rm Several remarks are to be made now:
\begin{enumerate}
    \item The solution $u^+$ constructed in this theorem can be used as a minimal solution in the class of nonnegative solutions.
    Indeed, assume that $v$ is another nonnegative solution with the same initial data $u_0\geq0$.
    Hence $v$ is also a supersolution and it can be used to play the role of $c_0\psi$.
    Thus, $u_n^+\leq v$ and passing to the limit we get $u\leq v$.
    \item There is a gap between the uniqueness and existence theorems: one requires that $u\ll\psi$ to get uniqueness
    while we are only able to construct a solution comparable to $\psi$. Thus, if one wants a global result of existence and uniqueness,
    one shall use two functions $\psi_1\ll\psi_2$ satisfying both \eqref{eq:def.psi}:
    there exists $u\leq \psi_1$ and it is unique, because $u\leq \psi_1\ll \psi_2$.
\end{enumerate}
\end{remark}

We end this section with a first estimate of the initial trace, using the same construction as for the existence theorem.
Notice that this estimate only concerns nonnegative (or nonpositive) initial data, but since
$|u_0|=u_0^+ - u_0^-$\,, then we shall be able to use the argument to obtain some precise estimates
also for changing sign solutions in the next section.

\begin{corollary}\label{cor:est.trace}
    Let $u$ be a nonnegative solution of \eqref{eq:0}--\,\eqref{eq:initial.data}.
    Then the initial trace $u_0=u(\cdot,0^+)\geq0$ of $u$
    satisfies the estimate:
    $$\text{For any }n\geq1,\ (J^{\ast n}\ast u_0)(x) <\infty\,.$$
\end{corollary}

\begin{proof}
    We use the same construction as above, considering $u_0\chi_n$ as initial data. The sequence $u_n$ of solutions with initial data $u_0\chi_n$ is monotone
    nondecreasing (recall that here $u_0\geq0$). Moreover, since $u_0\chi_n\in\L^2(\R^N)$ the representation formula holds:
    $$u_n(x,t)=\e^{-t}u_0(x)\chi_n(x)+\big(\omega(t)\ast(u_0\chi_n)\big)(x)\,.$$
    Since there exists a solution $u$, then the minimal solution $\underline{u}$ constructed as in the previous proof, as the limit of $u_n$, has to be
    finite almost everywhere. This implies that
    $$\big(\omega(t)\ast u_0\big)(x)<\infty\quad a.e.\ \text{in } \R^N\,,$$
    and actually everywhere since $\omega$ is continuous.
    Using the explicit formula for $\omega$, see \eqref{eq:def.omega}, we deduce that for any $n\geq1$, we have necessarily
    $J^{\ast n}\ast u_0<\infty$ everywhere in $\R^N$.
\end{proof}

\section{Optimal behaviour of the initial data for slow decaying kernels}
\setcounter{equation}{0}
\label{sect:optimal}

Using the results of the previous section we are able to
find the  optimal class of initial data such that there exists a
unique solution to~\eqref{eq:0}--\,\eqref{eq:initial.data}.
As we pointed out in the Introduction, the main point here is that $J$ decays slow at infinity. In fact, we will consider three different $J$'s, all of them decaying at infinity slower than $\exp(- \alpha|x|)$, $\alpha>0$. All three cases are treated in the same way: using elliptic barriers. That is, the
main point consists in finding the biggest function $f$ that verifies hypothesis~\eqref{eq:J.f}. The existence
of such a function implies the existence of the supersolution $\psi$.

\subsection{Power-type kernels}
\label{subsect:power}
As mentioned in the Introduction, we recover in this subsection the existence and uniqueness results in~\cite{AI} for power-type kernels. We also give complementary results concerning nonexistence of solutions and optimality of the barriers.

Let us suppose that
\begin{equation}\label{eq:j:momento}
\gamma_0:=\sup\Big\{\gamma>0: \int J(y)|y|^\gamma\d y< \infty\Big\}\in
(0,\infty)\,.
\end{equation}
That is, let us assume that $J$ has momentum of order $\gamma$ for
all $\gamma< \gamma_0$. Observe that, by continuity of the
function $\gamma \mapsto \int J(y)|y|^{\gamma}\d y$, we get that the
momentum of order $\gamma_0$ is infinite.
Kernels satisfying condition~\eqref{eq:j:momento} are for example those with potential decay $J(x)\sim|x|^{-(\gamma_0+N)}$
for $|x|$ large.

\begin{lemma}  \label{lemma.super.power}
  Let $J$ be a kernel satisfying~\eqref{eq:j:momento}. For any $\gamma<\gamma_0$ there exists $\lambda=\lambda(J,\gamma)$ such that $J$ verifies~\eqref{eq:J.f} with $f(x)=1+|x|^\gamma$.
\end{lemma}

\begin{proof}
Since $|x-y|\le 2 \max \{ |x|,\, |y|\},$ we have
$$
    \begin{aligned}
       J\ast (1+|x|^\gamma)-(1+|x|^\gamma)& = J\ast |x|^\gamma-|x|^\gamma\\& \le J\ast |x|^\gamma= \int_{\mathbb{R}^N} J(y)\,|x-y|^\gamma\d y\\
       &=\int_{B(0,|x|)}J(y) |x-y|^\gamma\d y+ \int_{\mathbb R^N\setminus B(0,|x|)}J(y) |x-y|^\gamma\d y\\
        &\le \displaystyle  |x|^\gamma 2^\gamma\int_{B(0,|x|)}J(y) \d y + 2^\gamma\int_{\mathbb R^N\setminus B(0,|x|)} J(y)|y|^\gamma\d y\\
        &\le \displaystyle \lambda(|x|^\gamma + 1),
\end{aligned}
$$
which holds if $\lambda\geq2^\gamma \max\Big\{1, \int_{\mathbb R^N} J(y)|y|^\gamma\d y\Big\}$.
\end{proof}

\begin{theorem}
    \label{thm:powertype}
    Let $J$ be a kernel satisfying~\eqref{eq:j:momento}. Let $u_0$ be a locally bounded initial data
    such that $|u_0(x)|\le c_0(1+|x|^\gamma)$ for
     $c_0>0$ and $0<\gamma<\gamma_0$. Then,    there exists a  global solution of \eqref{eq:0}--\,\eqref{eq:initial.data}, which satisfies
    $$
    |u(x,t)|\ll c_0 (1+|x|^{\gamma}) \mbox{ locally uniformly in }\ [0,\infty).
    $$
    Moreover, $u$ is the unique solution in this class.
\end{theorem}

\begin{proof}
It follows directly from theorems~\ref{thm:uniqueness},~\ref{thm:ex} and Lemma~\ref{lemma.super.power} by choosing
$\psi=c_0 \e^{\lambda t}(1+|x|^\gamma)$.
\end{proof}

\begin{theorem}
Let $J$ be a kernel satisfying~\eqref{eq:j:momento}. Assume that $u_0(x)\ge c_0(1+|x|^\gamma)$ for $c_0>0$ and $\gamma\ge \gamma_0$. Then there exists no nonnegative solution of   \eqref{eq:0}--\,\eqref{eq:initial.data}.
\end{theorem}

\begin{proof}
By Corollary \ref{cor:est.trace}, if there exists a nonnegative solution, then $J\ast u_0<\infty$, which is a contradiction with the growth of $u_0$.
\end{proof}

\begin{remark}\label{rem.nonex}
  {\rm If $u_0$ is a changing sign initial data, the non-existence theorem holds in the class of solutions, whose negative (or positive) part  is somehow controlled. More precisely, if we consider initial data $u_0$ such that
  $J*(u_0)_-< \infty$ and  $J*(u_0)_+= \infty$, then there are no solutions in the class $\{u_-\leq \psi\}$. Indeed, assuming there exists a solution $u$ in this class, then we can consider the nonnegative supersolution $v=u+\psi$, which allows us to construct the minimal solution $\underline u$ with initial data $v_0=u_0+\psi(0)$ as in Theorem~\ref{thm:ex}. But then, $J\ast v_0$ has to be finite, which contradicts the hypothesis on the initial data since of course $J\ast\psi(0)<\infty$.}
\end{remark}

\subsection{Exponential-type kernels}

If the initial data is an exponential type one, the analogous of
having momentum of order $\gamma$, see~\eqref{eq:j:momento}, is
\begin{equation}
  \label{eq:momento.exponencial}
  \gamma_0=\sup\Big\{\gamma>0: \int J(y)\e^{\gamma|y|}\d y< \infty\Big\}\in (0,\infty).
\end{equation}
For instance, this is the case for $J(x)\sim c_0 \e^{-\gamma_0 |x|}$.

\begin{lemma} \label{lemma.super.exponential}
Let $J$ be a kernel satisfying~\eqref{eq:momento.exponencial}. For any $\gamma<\gamma_0$ there exists $\lambda=\lambda(J,\gamma)$ such that $J$ verifies~\eqref{eq:J.f} with $f(x)= \e^{\gamma|x|}$.
\end{lemma}

\begin{proof}
   The proof follows as before. We only note that $|x-y|-|x|\le |y|$, then
    \[
        \begin{aligned}
        J\ast \e^{\gamma|x|}-\e^{\gamma|x|}&\le J\ast \e^{\gamma |x|}=\e^{\gamma|x|}\int_{\mathbb{R}^N} J(y) \e^{\gamma|x-y|-\gamma|x|} \,dy\\
        &\le \e^{\gamma|x|} \int_{\mathbb{R}^N} J(y) \e^{\gamma|y|} \,dy\\
        &\le \lambda\e^{\gamma|x|} ,
    \end{aligned}
    \]
    which is true, if   $\lambda\geq\int_{\mathbb{R}^N} J(y)\e^{\gamma|y|}\d y$.
\end{proof}

Hence, following the same arguments given in the previous subsection, with  $\psi(x,t)=\e^{\lambda t}\e^{\gamma|x|}$, we obtain the following results:

\begin{theorem}
        \label{thm:exponentialtype}
    Let $J$ be a kernel satisfying~\eqref{eq:momento.exponencial}. Let $u_0$ be a locally bounded initial data
    such that $|u_0(x)|\le c_0\e^{\gamma|x|}$ for
     $c_0>0$ and $0<\gamma<\gamma_0$. Then,    there exists a  global solution of \eqref{eq:0}--\,\eqref{eq:initial.data}, which satisfies
    $$
    |u(x,t)|\ll c_0\e^{\gamma|x|} \mbox{ locally uniformly in }\ [0,\infty).
    $$
    Moreover, $u$ is the unique solution in this class.
\end{theorem}

\begin{theorem}
Let $J$ be a kernel satisfying~\eqref{eq:momento.exponencial}. Assume that $u_0(x)\ge c_0\e^{\gamma|x|}$ for $c_0>0$ and $\gamma\ge \gamma_0$. Then there exists no nonnegative solution of   \eqref{eq:0}--\,\eqref{eq:initial.data}.
\end{theorem}

Moreover, from Remark \ref{rem.nonex} we also have nonexistence for solutions with changing sign if we control the negative part of the initial data.

\subsection{Tempered $\alpha$-stable type kernels}
\label{subsect:alpha.stable}
Assume now that  the kernel $J$  verifies
\begin{eqnarray}
  \label{eq:momento.power-exp}
      \gamma_0&=&\sup\Big\{\gamma>0: \int J(y)\e^{\gamma|y|}\d y< \infty\Big\}\in (0,\infty),\\
      \alpha_0&=&\sup\Big\{\alpha>0: \int J(y)\e^{\gamma_0|y|}(1+|y|)^{N+\alpha}\d y< \infty\Big\}\in (0,\infty).
   \label{eq:momento.power-expII}
  \end{eqnarray}
This is the case for instance if we consider $J(x)\sim\dfrac{\e^{-\gamma_0|x|}}{(1+|x|)^{N+\alpha_0}}$ as $|x|\to\infty$.

It is easy to check that $f(x)=\e^{\gamma|x|}$, with $\gamma<\gamma_0$ is a barrier, but it is not optimal as the following lemma shows:

\begin{lemma}  \label{lemma.super.power-exp}
Let $J$ be a kernel satisfying~\eqref{eq:momento.power-exp}. For any $\alpha<\alpha_0$ there exists $\lambda=\lambda(J,\alpha)$ such that $J$ verifies~\eqref{eq:J.f} with $f(x)=\e^{\gamma_0|x|}(1+|x|)^{N+\alpha}$.
\end{lemma}

The proof follows the same lines as in Subsection~\ref{subsect:power}. We prefer to write here the power as $N+\alpha$ instead of $\alpha$, in order to be more consistent with the usual notation for tempered $\alpha$-stable laws.

\begin{proof}
Let us check that $f(x)=\e^{\gamma_0|x|}(1+|x|)^{N+\alpha}$ is in fact a barrier:
$$
    J*\e^{\gamma_0|x|}(1+|x|)^{N+\alpha}-\e^{\gamma_0|x|}(1+|x|)^{N+\alpha}\leq \lambda \e^{\gamma_0|x|}(1+|x|)^{N+\alpha},
$$
for some $\lambda$ to be chosen later.
Indeed, Since $f$ is a positive function and  $|x-y|-|x|\le |y|$  we have that
$$
    \begin{aligned}
J\ast f-f\le J\ast f &= \int_{\mathbb{R}^N}J(y) \e^{\gamma_0|x-y|}(1+|x-y|)^{N+\alpha} \d y\\
            &= \e^{\gamma_0|x|} \int_{\mathbb{R}^N}J(y) \e^{\gamma_0|x-y|-\gamma_0 |x|} (1+|x-y|)^{N+\alpha}\d y\\
            &\leq \e^{\gamma_0|x|}\int_{\mathbb{R}^N}J(y) \e^{\gamma_0|y|} (1+|x-y|)^{N+\alpha}\d y =\e^{\gamma_0|x|} I.
    \end{aligned}
 $$
 In order to treat the power term, we note that $1\le (1+|x-y|)\le 2 (1+\max \{ |x|,\, |y|\})$, then
$$
    \begin{aligned}
 I&= \int_{B(0,|x|)}J(y) \e^{\gamma_0|y|} (1+|x-y|)^{N+\alpha} \d y +\int_{\mathbb{R}^N\setminus B(0,|x|)}J(y) \e^{\gamma_0|y|} (1+|x-y|)^{N+\alpha}\d y \\
&\le 2^{N+\alpha}(1+|x|)^{N+\alpha} \int_{B(0,|x|)}J(y) \e^{\gamma_0|y|} \d y+ 2^{N+\alpha}\int_{\mathbb{R}^N\setminus B(0,|x|)}J(y) \e^{\gamma_0|y|}(1+|y|)^{N+\alpha}\d y\\
&\le 2^{N+\alpha}(1+|x|)^{N+\alpha}\Big( \int_{B(0,|x|)}J(y) \e^{\gamma_0|y|} \d y+\int_{\mathbb{R}^N\setminus B(0,|x|)}J(y) \e^{\gamma_0|y|}(1+|y|)^{N+\alpha}\d y\Big).
    \end{aligned}
 $$
  Hence, using~\eqref{eq:momento.power-expII} it is enough to take $\lambda\geq 2^{N+\alpha}\int_{\mathbb{R}^N}J(y)\e^{\gamma_0 |y|}(1+|y|)^{N+\alpha}\d y$ to get the result.
\end{proof}

The results of existence, uniqueness and non-uniqueness are straightforward.
\begin{theorem}
        \label{thm:exponential-power-type}
    Let $J$ be a kernel satisfying~\eqref{eq:momento.power-exp}. Let $u_0$ be a locally bounded initial data
    such that $|u_0(x)|\le c_0\e^{\gamma|x|}(1+|x|)^{N+\alpha}$ for
     $c_0>0$, $0<\gamma<\gamma_0$ and $0<\alpha<\alpha_0$ . Then,    there exists a  global solution of \eqref{eq:0}--\,\eqref{eq:initial.data}, which satisfies
    $$
    |u(x,t)|\ll c_0\e^{\gamma|x|}(1+|x|)^{N+\alpha} \mbox{ locally uniformly in }\ [0,\infty).
    $$
    Moreover, $u$ is the unique solution in this class.
\end{theorem}

\begin{theorem}
Let $J$ be a kernel satisfying~\eqref{eq:momento.power-exp}. Assume that $u_0(x)\ge c_0\e^{\gamma|x|}(1+|x|)^{N+\alpha} $ for $c_0>0$, $\gamma\ge \gamma_0$ and $\alpha\ge \alpha_0$. Then there exists no nonnegative solution of   \eqref{eq:0}--\,\eqref{eq:initial.data}.
\end{theorem}

\section{Optimal behaviour of the initial for fast decaying kernels}
\setcounter{equation}{0}
\label{sect:optimal.II}

We face now the cases of fast decaying kernels. Although in these cases the constructions of barriers is also possible, we do not get the optimal behaviour and we need to develop a different approach.

The main point here consists in estimating the nonlocal heat kernel, $$\omega(x,t)=\e^{-t}\sum_{n=1}^\infty\dfrac{t^n}{n!}J^{\ast n}(x),$$
 associated to the equation, in order to get the optimal behaviour for the initial data.

\subsection{Compactly supported kernels}
\label{subsec:compact}

In this section we assume that $J$ is radially symmetric and compactly supported in $B_\rho$
for some $\rho>0$.

\begin{proposition}\label{prop:est.omega.compact}
    Let $J$ be a kernel such that $\mathrm{supp}(J)=B_\rho$ for some $\rho>0$. Then there exists two constants $c_3,c_4>0$
	and for any $0<\sigma<\rho$ there exist two constant $c_1,c_2$ depending on $\sigma$ such that
    \begin{equation}\label{est.compact}
    c_1\e^{-t}\e^{-(1/\sigma)|x|\ln |x|}\e^{(c_2+\ln t)|x|}\leq\omega(x,t)\leq c_3\e^{-t}\e^{-(1/\rho)|x|\ln |x|}\e^{(c_4+\ln t)|x|}\,.
    \end{equation}
\end{proposition}

We begin with the upper estimate:

\begin{lemma}\label{lem:omega.upper}
    Let $J$ be a kernel such that $\mathrm{supp}(J)=B_\rho$ for some $\rho>0$.
	Then there exist $c_3,c_4>0$ such that the following upper estimate holds:
    \begin{equation*}
   \sum_{n=1}^\infty\dfrac{t^n}{n!}J^{\ast n}(x)\leq c_3\e^{-(1/\rho)|x|\ln |x| + (c_4+\ln t)|x|}\,.
    \end{equation*}
\end{lemma}

\begin{proof}
    First it is clear that since $J$ has unit mass, then $\|J^{\ast n}\|_\infty\leq \|J\|_\infty=c>0$.
    Second, since the support of $J^{\ast n}$ is exactly $B_{(n+1)\rho}$, then it is enough
    to estimate, for fixed $t>0$, the function
    $$
        f(x):=\sum_{n\ge\big[|x|/\rho\big]}\frac{t^n}{n!}
    $$
    where $[r]$ stands for the entire part of $r>0$, that is, the greatest integer below (or equal to) $r$.
    Using the Taylor expansion of the exponential, we get that
    \begin{equation}\label{eq:f(x)}
        \sum_{n\ge K}\frac{t^n}{n!}=\int_0^t\frac{(t-s)^K}{K!}\e^s\d s\le \frac{\e^t t^K}{K!}\,.
    \end{equation}
    Then we use Stirling's formula for the factorial: $K!=\sqrt{2\pi K}(K/{\rm e})^K(1+O(1/K))$
    to obtain that there exists $c'>0$ such that:
    $$
        \sum_{n\ge K}\frac{t^n}{n!}\le c' \e^t t^K K^{-1/2}\Big(\frac{{\rm e}}{K}\Big)^K\,.
    $$
    Coming back with $K=\big[|x|/\rho\big]$ we then obtain:
    $$
        \sum_{n\ge\big[|x|/\rho\big]}\frac{t^n}{n!}\leq cc't^{[|x|/\rho]}\e^{-(1/2)\ln[|x|/\rho]} e^{-[|x|/\rho](\ln [|x|/\rho]-1)}\,.
    $$
    Then it follows that for some $c_3,c_4>0$,
    $$
       \sum_{n=1}^\infty\dfrac{t^n}{n!}J^{\ast n}(x)\leq c_3 \e^{-(1/\rho) |x|\ln |x| + (c_4+\ln t)|x|/\rho}\,,
    $$
    which is the desired estimate.
\end{proof}

\begin{lemma}\label{lem:J*.lower}
    For any $0<\sigma<\rho$, there exist $c>0$ and $\mu\in(0,1)$ depending only on $J$ and $\sigma$ such that
    \begin{equation}
    \label{eq:J*.lower}
        \forall n\geq1\,,\ \forall\,x\in B_{n\sigma}\,,\quad J^{\ast n}(x)\geq c\mu^n\,.
    \end{equation}
\end{lemma}

\begin{proof}
	For some technical reason, we need to introduce $\sigma'=(\sigma+\rho)/2\in(\sigma,\rho)$. Since $J$ is symmetric and compactly
	supported in $B_\rho$, noting $y=(y_1,\dots,y_N)$ we set
    $$\mu:=\int_{\{\sigma'<y_N<\rho\}}J(y)\d y>0\,,$$
    which is invariant under rotations of the coordinates.
	Recall also that $\supp(J^{\ast n})=B_{n\rho}$ which strictly contains $B_{n\sigma}$,
    so that for $n_0\geq1$ to be fixed later, we can define
    $$
    c:=\min_{B_{n_0\sigma}} J^{\ast n_0}/\mu>0\,.
    $$
    With this choice, property \eqref{eq:J*.lower} is valid for $n=n_0$
    and we shall show the property by induction for $n>n_0$.
    So we us assume that in $B_{n\sigma}$, $J^{\ast n}\geq c\mu^n$ and
    for $x\in B_{(n+1)\sigma}$ let us estimate:
    $$
        J^{\ast(n+1)}(x)=\int_{\R^N}J(x-y)J^{\ast n}(y)\d y=\int_{\R^N}J(y)J^{\ast n}(x-y)\d y\,.
    $$
    We assume without loss of generality that $x=(0,\dots,0,x_N)$
    $$
        J^{\ast(n+1)}(x)\geq\int_{\{ \sigma'<y_N<\rho \}}J(y)J^{\ast n}(x-y)\d y\,.
    $$
	Then since $|x|<(n+1)\sigma$, if $y$ satisfies $\sigma'<y_N<\rho$, we have:
    $$
        \begin{aligned}
        |x-y|^2= & |x_N-y_N|^2+\sum_{i=1}^{(N-1)}|y_i|^2\,,\\
        \leq & \Big((n+1)\sigma-\sigma'\Big)^2+\rho^2\,,\\
        \leq & (n\sigma+\sigma-\sigma')^2+\rho^2\,.\\
        \end{aligned}
    $$
	Since $\sigma<\sigma'$ we have for $n$ big enough that $(n\sigma+\sigma-\sigma')^2+\rho^2\leq(n\sigma)^2$
	and it follows that for such $x,y$, $|x-y|\leq n\sigma$. Hence we fix $n_0=n_0(\sigma,\sigma',\rho)$
	such that for any $n\geq n_0$ and $x\in B_{(n+1)\sigma}$ we have:
    $$
        J^{\ast(n+1)}(x)\geq\int_{\{ \sigma'<y_N<\rho \}}(c\mu^n)J(y)\d y=c\mu^{n+1}\,.
    $$
	Thus the property is proved by induction for $n>n_0$. Finally, the cases when $n\leq n_0$ only concern a finite number
	of terms so that, up to redefining $c$ and $\mu$ taking the min over all the first terms, we obtain~\eqref{eq:J*.lower}
	for any $n\geq1$.
\end{proof}

\begin{proof}[Proof of Proposition \ref{prop:est.omega.compact}]
    The upper estimate follows directly from Lemma \ref{lem:omega.upper}, so we turn to the lower estimate. We
    first notice that for any fixed $n$
    $$\omega(x,t)\geq \e^{-t}\frac{t^nJ^{\ast n}(x)}{n!}\,.$$
	If we take $n=[|x|/\sigma]-1$ then $x\in B_{n\sigma}$ and we may use
    Lemma~\ref{lem:J*.lower}:  there exists $c,\mu>0$ depending on $\sigma$ such that
    $$\omega(x,t)\geq c\e^{-t}\frac{(\mu t)^{[|x|/\sigma]}}{[|x|/\sigma]!}\,.$$
    Thus, using again Stirling's formula for the factorial, we obtain that for some $c,c'>0$:
    $$\omega(x,t)\geq c\e^{-t}\e^{-(1/\sigma)|x|\ln |x|+(c'+\ln (\mu t)|x|}\,.$$
    Hence, we have obtained indeed that there exists $c_1,c_2>0$ such that
    $$\omega(x,t)\geq c_1\e^{-t}\e^{-(1/\sigma)|x|\ln |x|}\e^{(c_2+\ln t)|x|}\,,$$
    which ends the proof.
\end{proof}

\begin{theorem}
    Let $J$ be a radially symmetric kernel, compactly supported in $B_\rho$ for some $\rho>0$. Let $u_0$ be a locally bounded initial data and $c_0>0$. Then the following holds:\\[6pt]
    $(i)$ If $|u_0(x)|\leq c_0\e^{\alpha|x|\ln|x|}$ for some $0<\alpha<1/\rho$, then there exists a global solution of \eqref{eq:0}--\,\eqref{eq:initial.data}.\\[6pt]
    $(ii)$ If $u_0(x)\geq c_0\e^{\beta |x|\ln|x|}$ for some $\beta>1/\rho$, then there exists no nonnegative solution of \eqref{eq:0}--\,\eqref{eq:initial.data}.
\end{theorem}

\begin{proof}
We consider the sequence $\{u_n\}$, given by the solution of (\ref{eq:0}) with initial data $u_n(x,0)=u_0(x)\chi_n$. This functions are given by convolution:
$$
u_n(x,t)=\e^{-t}u_0(x)\chi_n(x)+(\omega(t)\ast u_0\chi_n)(x).
$$

\noindent$(i)$ -- If $|u_0(x)|\leq c_0\e^{\alpha|x|\ln|x|}$ for some $0<\alpha<1/\rho$, then $\omega(t)\ast u_0 $ is defined for any $t>0$, and the sequence $\{u_n\}$ is uniformly bounded by $\phi(x,t)=\e^{-t}u_0(x)+\omega(t)\ast u_0$. Therefore, following the same argument given in the proof of Theorem \ref{thm:ex} we get a global minimal solution, which is given by convolution.

\noindent$(ii)$ -- Let us assume that $u_0(x)\geq c_0\e^{\beta |x|\ln|x|}$ for some $\beta>1/\rho$, and that there exists a nonnegative solution $u$.
Then notice first that there exists also a $\sigma\in(0,\rho)$ such that $\beta>1/\sigma$ and either the negative or the positive part
of $u_0$ satisfies a similar estimate, for instance:
$$(u_0)_+(x)\geq c_0\e^{\beta |x|\ln|x|}\,.$$
We use an approximation by compactly supported initial data $(u_0)_+\chi_n$ and we see that by comparison, in $\R^N\times\R_+$ we have
$$\e^{-t}(u_0)_+\chi_n + \omega(t)\ast (u_0)_+\chi_n\leq |u|<\infty\,.$$
But on the other hand, since $\beta>1/\sigma$, $\omega(t)\ast (u_0)_+\chi_n\to +\infty$ as $n\to\infty$, which is a contradicton
with the bound above. The conclusion is that no solution can exist.

\end{proof}

\begin{remark}\rm Several remarks are to be made:
\begin{enumerate}
    \item Actually the estimates on $\omega$ allow us to treat more delicate situations like for instance (taking $\rho=1$ for simplicity),
    $u_0(x)\sim\e^{|x|\ln|x|-|x|\ln(\ln|x|)}$ which gives rise to a global solution.
	\item Finer estimates (in $x$) can be obtained by using a better expansion of the factorial -- which
    is known, but anyway, the $(c+\ln t)|x|$-terms cannot be avoided which implies an estimate of the order of $|x|$
    in the exponential.

	\item An interesting question concerns the critical growth, when $u_0(x)$ behaves like $\e^{(1/\rho)|x|\ln|x|}$.
	We shall see in the gaussian case below -- see Theorem \ref{thm:gaussian.ex} $(iii)$ -- that solutions with critical initial data
	blow up in finite time, as it is the case for the local heat equation.
	We are convinced that it is also the case for compactly supported kernels but the lower estimate of $\omega$ which implies
	a $\sigma$ strictly less than $\rho$ does not allow us to conclude.

    \item The same arguments as in Remark~\ref{rem.nonex} are valid in this case, replacing the condition on the initial data by $\omega(t)\ast (u_0)_-<\infty$ and $\omega(t)\ast(u_0)_+=\infty$ for any $t>0$.
\end{enumerate}
\end{remark}

Now we turn to uniqueness results. In order to do so, we use the fact that, given a function $f\geq0$, the function $\omega(t)\ast f$ is a supersolution of~\eqref{eq:0}.
We shall prove first that the constructed solution remains inside a suitable class of solutions, and then use the supersolution provided by
$\omega(t+1)*f$ in order to apply the comparison
principle and get uniqueness.

\begin{proposition}
    Let $J$ be a radially symmetric kernel, compactly supported in $B_\rho$ for some $\rho>0$. Let $u_0$ be a locally bounded initial data,  $\alpha\in(0,1/\rho)$ and $c_0>0$. If  $|u_0(x)|\leq c_0\e^{\alpha|x|\ln|x|}$, there exists $A>0$ such that
    $$|u(x,t)|\leq C\e^{\beta|x|\ln|x|}\,.$$
    for any $\beta\in(\alpha,1/\rho)$.
 \end{proposition}

\begin{proof}
Let us split again the initial data into its positive and negative part, $|u_0|=(u_0)_++(u_0)_-$. Then
$$
u^+(x,t)=\e^{-t}(u_0)_+(x)+\big(\omega(t)\ast (u_0)_+\big)(x)\le \e^{-t} c_0\e^{\alpha|x|\ln|x|}+ c_0  \omega(t)\ast e^{\alpha|x|\ln|x|}
$$
and the same holds for $u^-$.
For $t\in(0,T)$, we have the estimate $\omega(x,t)\leq c\e^{-(1/\rho)|x|\ln|x|+c(T)|x|}$\,,
    for some $c>0$ and $c(T)\in\R$. Hence,
$$
\begin{aligned}
    (\omega(t)\ast \e^{\alpha|x|\ln|x|})(x)&\leq c'\int_{\R^N}\e^{-(1/\rho)|y|\ln|y|+c'(T)|y|+\alpha|x-y|\ln|x-y|}\d y\,,\\
        &\leq c'\int_{\R^N}\e^{-(1/\rho)|y|\ln|y|+c'(T)|y|+\alpha|x|\ln|x|+\alpha|y|\ln|y|+\alpha|x|+\alpha|y|}\d y\,,\\
        & \leq c'\e^{\alpha|x|\ln|x|+\alpha|x|}\int_{\R^N}\e^{(\alpha-1/\rho)|y|\ln|y|+(c'(T)+\alpha)|y|}\d y\,,\\
        & \leq c''\e^{\alpha|x|\ln|x|+\alpha|x|}\,.
    \end{aligned}
$$
Summing up, we get that for all $\beta\in (\alpha,1/\rho)$ and $x$
large enough,
$$
u^+(x,t)\le C \e^{\alpha|x|\ln|x|+\alpha|x|}\le  \e^{\beta |x|\ln|x|} \quad\mbox{and}\quad  u^-(x,t)\le C e^{\alpha|x|\ln|x|+\alpha|x|}\le  \e^{\beta |x|\ln|x|}
$$
Finally, we note that as $T$ is arbitrary,
$$
|u(x,t)|\leq  C  \e^{\beta |x|\ln|x|}
$$  for all $t\ge 0$.

\end{proof}

In the following, we denote by $f_\beta$ the function $f_\beta(x)=\e^{\beta|x|\ln|x|}$.

\begin{lemma}\label{lem:psi}
Let $J$ be a radially symmetric kernel, compactly supported in $B_\rho$ for some $\rho>0$. Let $u_0$ be a locally bounded function, $\alpha\in(0,1/\rho)$ and $c_0>0$. If  $|u_0(x)|\leq c_0\e^{\alpha|x|\ln|x|}$, then for any $T>0$,
$\beta\in(\alpha,1/\rho)$ and $A>0$ large enough the function
$$
\psi(x,t):=A(\omega(t+1)\ast f_\beta)(x)
$$
is a supersolution  of \eqref{eq:0}--\,\eqref{eq:initial.data} on $\R^N\times[0,T]$.
\end{lemma}
\begin{proof}
    We assume here that $\rho=1$ in order to simplify the proof a little bit,
    the modifications being straightforward in the other cases.
    Since $\omega$ satisfies \eqref{eq:omega}, the function $\psi$ verifies
    $$\psi_t=J\ast \psi -\psi + A\e^{-(t+1)}J\ast f_\beta\geq J\ast \psi -\psi\,,$$
    so that we have a supersolution of the equation, and we need now to compare the initial data.

    We start from the fact that for $t\in(0,T)$, and any $\sigma>1$, we have the estimate
	$$\omega(x,t+1)\geq c\e^{-(1/\sigma)|x|\ln|x|+c(T)|x|}\,,$$
    for some $c>0$ and $c(T)\in\R$. Hence,
    $$\begin{aligned}
    (\omega(t+1)\ast f_\beta)(x)&\geq c\int_{\R^N}\e^{-(1/\sigma)|y|\ln|y|+c(T)|y|+\beta|x-y|\ln|x-y|}\d y\,,\\
        &\geq c\int_{\{ |x-y|>|x| \}}\e^{-(1/\sigma)|y|\ln|y|+c(T)|y|+\beta|x|\ln|x|}\d y\,,\\[2mm]
        & \geq c(x)\e^{\beta|x|\ln|x|}\,,
    \end{aligned}$$
    where the function $c(x)$ is given by
    $$
     c(x)=c\int_{\{|x-y|>|x|\}}\e^{-(1/\sigma)|y|\ln|y|+c(T)|y|}\d y\,.
    $$
	We claim that the constant $c(x)$ is uniformly bounded from below. Indeed,
	since $$|x-y|^2=|x|^2+|y|^2-2\sum x_iy_i\,,$$ the set $\{|x-y|>|x|\}$ contains at least
	$\{y_i x_i<0 \mbox{ for } i=1,\cdots,N\}$. Moreover, since the integrand is radial and integrable, we have
	$$
	c(x)\geq c\int_{\{\forall i,\ x_iy_i<0\}}\e^{-|y|\ln|y|+c(T)|y|}\d y=c\int_{\{\forall i,\ y_i>0\}}\e^{-|y|\ln|y|+c(T)|y|}\d y=c_1>0\,,
	$$
	and thus we get that
    \begin{equation}\label{eq:lower.behavior.psi}
        (\omega(x,t+1)\ast f_\beta)(x)\geq c_1\e^{\beta|x|\ln|x|}\,.
    \end{equation}
    This implies that if $A>c_0/c_1$\,, then $\psi(0)=(\omega(1)\ast f_\beta)\geq u_0$.
\end{proof}

\begin{theorem}
\label{thm:sop.compact.minimal}
    Let $J$ be a radially symmetric kernel, compactly supported in $B_\rho$ for some $\rho>0$. Then for any
	$0<\alpha<1/\rho$, there exists a unique solution of \eqref{eq:0}--\,\eqref{eq:initial.data}
	such that $|u(x,t)|\leq C(t)\e^{\alpha|x|\ln|x|}$, where $C(\cdot)$ is locally bounded in $[0,\infty)$.
\end{theorem}

\begin{proof}
	Notice first that by assumption, the initial data satifies $|u_0|\leq C(0)\e^{\alpha|x|\ln|x|}$
	so that we know how to construct a solution satisfying the assumptions of the theorem.
    Now, let us fix $\beta\in(\alpha,1/\rho)$ and consider the supersolution
    $$\psi(x,t):=A(\omega(t+1)\ast f_\beta)(x)\,,$$
    which grows at least as $\e^{\beta|x|\ln|x|}$ on $\R^N\times[0,T]$ -- see \eqref{eq:lower.behavior.psi}.
    Since this supersolution grows strictly faster than any solution satisfying $|u(x,t)|\leq C(t)\e^{\alpha|x|\ln|x|}$,
	we can apply Theorem \ref{thm:uniqueness} to get uniqueness.
\end{proof}

\subsection{Gaussian kernels}
\label{subsec:gauss}

We assume here that the $J$ is a centered gaussian with variance $\sigma^2=1/(2\gamma^2)$
(we write it under this form in order to simplify the statements), that is:
\begin{equation}\label{eq:def.gaussian}
    J(y)=c(\gamma)\e^{-\gamma^2|y|^2}\,,\ c(\gamma)=\Big(\int_{\R^N} \e^{-\gamma^2|y|^2}\d y\Big)^{-1}=
    \Big(\frac{\gamma^2}{\pi}\Big)^{N/2}\,.
\end{equation}
We use the same method as in the case of compactly supported kernels
to estimate the nonlocal heat kernel associated to the equation. Let us first give a slightly better estimate
than that of Lemma \ref{lem:J*.lower}:

\begin{lemma}\label{lem:J*.lower.gauss}
    There exist $c>0$ and $\mu\in(0,1)$ depending only on $J$ such that
    \begin{equation*}
        \forall n\geq1\,,\ \forall\,x\in B_{n}\,,\quad J^{\ast n}(x)\geq c\mu^n\,.
    \end{equation*}
\end{lemma}
\begin{proof}
	The method is exactly the same as for Lemma \ref{lem:J*.lower}, except that since $J$ is not compactly supported,
	it is even easier. In the present case we set
    $$\mu:=\int_{\{3/2<y_N<2\}}J(y)\d y>0\,,$$
	and it is enough to check that if $|x|<n$, and $3/2<y_N<2$, then
	$|x-y|^2\leq (n-1/2)^2+2\,.$
	Hence for $n$ big enough, we have for such $x,y$, that $|x-y|\leq n$ and we can proceed by induction as
	in Lemma \ref{lem:J*.lower}.
\end{proof}

\begin{proposition}\label{prop:est.omega.gaussian}
    Let $J$ be defined by \eqref{eq:def.gaussian} for some $\gamma>0$. Then there exist some positive
    constants $c_1,c_2,c_3,c_4>0$ depending only on $\gamma$ and $N$ such that for $|x|>1$,
    \begin{equation*}
    c_1\e^{-t}\e^{-\gamma|x|(\ln |x|)^{1/2}}\e^{(c_2+\ln t)|x|}\leq
    \omega(x,t)\leq c_3\e^{-\gamma|x|(\ln |x|)^{1/2}}\e^{(c_4+\ln t)|x|}\,.
    \end{equation*}
\end{proposition}

\begin{proof}
    Fixing first $|x|>1$, we begin by using the explicit formula for convolution of gaussian laws:
    $$J^{\ast n}(y)=c(\gamma)^n\Big(\frac{\gamma^2}{n\pi}\Big)^{N/2}\e^{-\gamma^2|y|^2/n}\,,$$
    where $c(\gamma)$ is defined in \eqref{eq:def.gaussian}.
    Then it follows that we can estimate from above
    $$\omega(x,t)=\e^{-t}\sum_{n=1}^\infty\frac{J^{\ast n}(x)t^n}{n!}\leq \Big(\frac{\gamma^2}{\pi}\Big)^{N/2}\e^{-t}
    \Big\{\e^{-\gamma^2|x|^2/K}\sum_{n=1}^K \frac{c(\gamma)^n t^n}{n!}
    + \sum_{n>K}\frac{c(\gamma)^n t^n}{n!}\Big\}\,,$$
    where $K>1$ is an integer to be fixed below, depending on $x$.
    Using \eqref{eq:f(x)} for the second term, we get that for some constant $C(\gamma,N)>0$ we have:
    $$\omega(x,t)\leq C(\gamma,N)\, (c(\gamma)t)^K\Big(\e^{-\gamma^2|x|^2/K}  + \frac{1}{K!}\Big)\,.$$

    Now, in order to optimize this estimate as $|x|\to\infty$, we choose $K$ such that both terms are comparable,
    which means as $|x|\to\infty$, using Stirling's formula:
    \begin{equation}\label{eq:gauss.K}
		\gamma^2|x|^2\sim K\ln (K!)\sim K^2\ln K\,.
	\end{equation}
	Taking logs, we get $\ln |x|\sim\ln K$ and then taking the square root in \eqref{eq:gauss.K} we obtain
    $$K\ln K\sim \gamma|x|(\ln|x|)^{1/2}\,.$$
    With this choice, the estimate becomes:
    $$
	\omega(x,t)\leq 2C(\gamma,N)\,(c(\gamma)t)^K\e^K\frac{1}{K!}\leq 2C(\gamma,N)\,(c(\gamma)t)^K
	\e^{- K\ln K + O(K)}\,,
	$$
    and we argue as in Lemma \ref{lem:omega.upper}. Notice that the coefficients $c_2$ and $c_4$
    include the constant $c(\gamma)$ since $\ln (c(\gamma)t)=\ln c(\gamma)+\ln t$.

	To get the lower bound, we first notice that for any fixed $n$,
	$$\omega(x,t)\geq \e^{-t}\frac{t^nJ^{\ast n}(x)}{n!}\,.$$
	Then we choose $n=[K(x)]$, the entire part of $K$ defined by \eqref{eq:gauss.K},
	we use Lemma \ref{lem:J*.lower.gauss}:
	$$\omega(x,t)\geq c\e^{-t}\frac{(\mu t)^{[K]}}{[K]!}\,,$$
	and we end up as in Proposition \ref{prop:est.omega.compact}.
\end{proof}

As a direct consequence we have:

\begin{theorem}\label{thm:gaussian.ex}
    Let $J$ be defined in \eqref{eq:def.gaussian} for some $\gamma>0$, let $u_0$ be
    a locally bounded initial data and $c_0>0$ arbitrary. Then the following holds {\rm (}we consider only $|x|>1${\rm):}\\[6pt]
    $(i)$ If $|u_0(x)|\leq c_0\e^{\alpha|x|(\ln|x|)^{1/2}}$ for some $\alpha<\gamma$, then there exists a global solution of \eqref{eq:0}.\\[6pt]
    $(ii)$ If $u_0(x)\geq c_0\e^{\beta|x|(\ln|x|)^{1/2}}$ for some $\beta>\gamma$, then there does not exist any nonnegative solution of~\eqref{eq:0}.\\[6pt]
    $(iii)$ If $u_0(x)=c_0\e^{\gamma|x|(\ln|x|)^{1/2}+f(|x|)}$ where for some $\alpha<0<\beta$, $\alpha s\leq f(s)\leq \beta s$,
    then the minimal solution of \eqref{eq:0} blows up in finite time.
\end{theorem}

\begin{proof}
	Points $(i)$ and $(ii)$ are done exactly as in the case of compactly supported kernels. It remains to show $(iii)$.
	Since $\alpha s\leq f(s)\leq \beta s$, then we know that at least, for $t>0$ small enough, the convolution $\omega(t)\ast u_0$ converges.
	More precisely, it is enough to choose $\beta + c_2 + \ln t <0$, hence $t<\e^{-\beta-c_4}$.
    On the other hand, if $t$ is big enough, using the estimate from below of $\omega(t)$, we know
    that the convolution blows up, at least for $t>\e^{-\alpha-c_2}$.
    Hence the solution is defined for a short time interval, but it eventually blows up in finite time.
\end{proof}

Concerning changing sign solutions and uniqueness, the results for the compactly supported case have a direct translation in the present situation.
The adaptations being straightforward, which essentially consist in changing $\ln(|x|)$ to $\ln(|x|)^{1/2}$ for $|x|>1$, we skip the details.

In order to avoid the problem with the log if $|x|<1$, we consider the function
\begin{equation*}
    g_\alpha(x):=\e^{\alpha|x|(\ln(|x|+1))^{1/2}}\,.
\end{equation*}

Everything follows from the estimates in Lemma \ref{lem:psi} which are valid in the following form here:
\begin{lemma}
    Let $\alpha\in(0,\gamma)$ and let $u_0$ be a nonnegative continuous
    function such that $$u_0(x)\leq c_0\e^{\alpha|x|(\ln(1+|x|))^{1/2}}\,.$$
    Then for any $T>0$, if $A>0$ is big enough the function $\psi(x,t):=A(\omega(t+1)\ast g_\alpha)(x)$ is a supersolution
    of \eqref{eq:0}--\,\eqref{eq:initial.data} on $\R^N\times[0,T]$.
    Moreover, for $t\in[0,T]$, there exists a constant $A'>0$ such that
    $$\psi(x,t)\leq A'\e^{\alpha|x|(\ln(1+|x|))^{1/2}+\alpha|x|}\,.$$
\end{lemma}
\begin{proof}
    The proof follows exactly that of Lemma \ref{lem:psi}, the only modification we need
    is provided by the inequality $(a+b)^{1/2}\leq a^{1/2}+b^{1/2}$
    for the upper estimate, applied to $(\ln(1+|x|) + \ln(1+|y|))^{1/2}$.
\end{proof}

With this Lemma, Theorem \ref{thm:sop.compact.minimal} can be written here as follows, with obvious adaptations:
\begin{theorem}\label{thm:gaussian}
    Let $J$ be defined in \eqref{eq:def.gaussian} for some $\gamma>0$. Then for any
	$0<\alpha<\gamma$, there exists a unique solution of \eqref{eq:0}--\,\eqref{eq:initial.data}
	such that $|u(x,t)|\leq C(t)\e^{\alpha|x|\ln|x|}$, where $C(\cdot)$ is locally bounded in $[0,\infty)$.
\end{theorem}

\section{Explicit solutions. Asymptotic behaviour}
\label{sect:explicit}
\setcounter{equation}{0}

Let $J$ be a kernel such that
\begin{equation}
\label{eq:j:momento2p}
\gamma_0:=\sup\big\{\gamma>0 :\int J(y)|y|^\gamma \d y<\infty\big\}\in(0,\infty)
\end{equation}
Since $J$ is symmetric it is easy to see that for any integer $p\in (0, \gamma_0)$ the following computation makes sense:
\begin{equation*}
   \begin{aligned}
      J*|x|^{2p}-|x|^{2p}&=\int J(y)(|x-y|^{2p}-|x|^{2p})\, \d y\\
                         &=\sum_{i=0}^{p-1}{{2p}\choose{2i}}|x|^{2i}\int J(y)|y|^{2(p-i)}\,\d y=\sum_{i=0}^{p-1}{{2p}\choose{2i}}m_{2(p-i)}|x|^{2i}.
    \end{aligned}
\end{equation*}
With this expression in mind we prove the following theorem:
\begin{theorem}
  \label{thm:explicit}
  Under hypothesis~\eqref{eq:j:momento2p}, for any integer $p\in (0, \gamma_0)$  the unique solution of~\eqref{eq:0} with initial data $u_0(x)=|x|^{2p}$ has the explicit form
    \begin{equation}
      \label{eq:u:explicit}
      u(x,t)=|x|^{2p}+\sum_{k=1}^p c_k(x)\frac{t^k}{k},
    \end{equation}
    where
    \begin{equation}
      \label{eq:recursive:coef}
        c_1(x)=J*|x|^{2p}-|x|^{2p},\qquad c_k(x)=\frac{1}{(k-1)}\big(J* c_{k-1}(x)-c_{k-1}(x)\big).
    \end{equation}
\end{theorem}

\begin{proof}
    We first check that $u$ given by~\eqref{eq:u:explicit} is a solution of~\eqref{eq:0} if the coefficients $c_k$ are given recursively by~\eqref{eq:recursive:coef}.
    Indeed,
    \[
    u_t=\sum_{k=1}^p c_k(x)t^{k-1},\qquad\mbox{and}\qquad J*u-u=J*u_0-u_0+\sum_{k=1}^p \frac{t^k}{k}(J*c_k(x)-c_k(x)),
    \]
    form where we get that $u$ satisfies the equation by replacing the coefficients by its recursive expression.
     Notice also that $u(x,0)=u_0(x)$.
\end{proof}

\noindent{\bf Examples.} Consider the initial data $u_0=|x|^2$. The solution $u$ of~\eqref{eq:0} has the explicit form
    \[
    u(x,t)=|x|^2+m_2t.
    \]
For $u_0=|x|^4$ we have
    \[
    u(x,t)=|x|^4+\frac{1}{2}m_2^2t^2+(m_4+2m_3|x|^2)t.
    \]

\begin{remark}
  {\rm It follows form~\eqref{eq:u:explicit} that $u(x,t)=|x|^{2p}+t^p+o(t^p)$ as $t\to\infty$ which is the same relation between space and time as for the heat equation.}
\end{remark}

\section{Extensions and comments}
 \label{sect:extension}
\setcounter{equation}{0}

\subsection{Back to the initial trace}

There is a huge difference between the nonlocal case treated here and the heat equation. Strong nonnegative solutions have an initial trace which is necessarily a locally integrable function, not a general measure. This has to do with the absence of regularization of this nonlocal equation.

However it is not clear what happens for changing sign solutions, since the monotonicity property that was used in Proposition~\ref{prop:trace} is not valid here. In the case of the usual heat equation, it is possible to construct a dipole solution with initial data $u_0=(\delta_0)'$. A similar construction can be done here,  but the associated solution maintains the derivative of the delta measure, as can be easily seen passing to the limit in the explicit formula for the solution. This is a hint that probably the initial trace of strong changing sign solutions is also a locally integrable function, but we are not able to prove it so far.

 \subsection{Singular kernels}\label{subsect:singular}

In \cite{AI}, the authors address the case of the fractional Laplacian, $J(z)=1/|z|^{N+\alpha}$, from the point of view of
unbounded solutions. To be precise, they consider the more general equation $u_t+H(t,x,u,Du)+g[u]=0$.
The Hamiltonian $H$ has to satisfy some structure assumption and the nonlocal term $g[u]$ is given by
\begin{equation}\label{eq:levy}
    g[u](x):=-\int_{\R^N}\Big\{u(x+z)-u(x)-\frac{Du(x)\cdot z}{1+|z|^2}\Big\}\d\mu(z)\,,
\end{equation}
where $\mu$ is a symmetric L\'{e}vy measure with density $J$ satisfying:
$\int \inf(|z|^2,1)J(z)\d z<\infty$.
Of course this equation has to be understood in the sense of viscosity solutions.

We already pointed out that, as far as finding the optimal behaviour of the initial data is
concerned, only the tail of $J$ plays an important role, not its singularity at the origin.
This is reflected in \cite{AI} in the fact that
the integrability condition only concerns the bounded part $\mu_b$ of the L\'evy measure (defined as the
restriction of $\mu$ to the complement of the unit ball). Thus, our various
estimates for the optimal data remain valid even for singular measures.

For instance, if we consider the equation
$$u_t(x,t)=\int_{\R^N}\big\{u(x+z)-u(x)\big\}J(z)\d z\,,$$
where $J$ may be singular at the origin and behaves like a power for large $|z|$,
then existence and uniqueness of solutions is based
on the existence of a function $f$ such that
$$
   \mathcal{L}f(x):=\int_{\R^N}\Big\{f(x+z)-f(x)\Big\}J(z)\d z\leq \lambda f(x)\,.
$$
This is what is done in the construction of barriers in \cite{AI}. Here is also another explicit situation:
following Section \ref{subsect:alpha.stable}, let the kernel $J$ be a tempered $\alpha$-stable law
with $0<\alpha<2$, often used in finance modeling:
$$J(x)=\frac{\e^{-|x|}}{|x|^{N+\alpha}}\,.$$
Then for any $\gamma<1$, the function $f(x)=\e^{\gamma |x|}$ is a barrier. Indeed,
$$
\begin{aligned}
 \mathcal{L}f(x) & =\int_{\R^N}\frac{\e^{\gamma|x+z|}-\e^{\gamma|x|}}{|z|^{N+\alpha}}\e^{-|z|}\d z=
\e^{\gamma |x|}\int_{\R^N} \frac{\e^{\gamma|x+z|-\gamma|x|}-1}{|z|^{N+\alpha}}\e^{-|z|}\d z\\
   &\leq\e^{\gamma |x|}\Big(\int_{\{|z|\leq \delta\}} \frac{\e^{\gamma |z|}-1}{|z|^{N+\alpha}}\e^{-|z|}\d z+
\int_{\{|z|> \delta\}} \frac{\e^{\gamma |z|}-1}{|z|^{N+\alpha}}\e^{-|z|}\d z\Big)
   \\
   &\leq\e^{\gamma |x|}(C_1+C_2)=\lambda f(x)\,.
\end{aligned}
$$

Notice that to give sense to those integrals near $z=0$, either one has to take $\alpha<1$, or to
understand the operator as the principal value of the integral if $1\leq\alpha<2$.

A similar computation shows that for any $0<\alpha'<\alpha$, the function $f(x)=\e^{|x|}(1+|x|)^{\alpha'}$ is also
a barrier. Then, existence and uniqueness results that complement those of \cite{AI}
can be obtained provided the initial data satisfies $|u_0(x)|\leq\e^{|x|}(1+|x|)^{\alpha'}$ for some $0<\alpha'<\alpha$.

In the introduction of \cite{AI}, it is conjectured that it may always be possible to construct some
barrier for any given L\'evy measure. This may be true, however we point out again that
in the case of fast decaying (or compactly supported) kernels, the elliptic barriers do not give the optimal behaviour.

Typically, if we think about the singular measure with density
$$\mu(z)=\frac{\mathbf{1}_{\{|z|<1\}}}{|z|^{N+\alpha}}\,,$$
then the barriers that can be constructed allow to get existence only up to an $\exp{(|x|)}$-behaviour, but they miss
the $\exp(|x|\ln|x|)$-behaviour.

We finish this section by mentioning that a way to get optimal existence for singular measures is to
to approach the singular measure with a monotone sequence of bounded measures. Passing to the limit has to be done in the
viscosity sense of course, and this allows to get existence of a solution for a general Lévy measure $\mu$ provided the initial data
satisfies the estimates for $\mu_b$, the nonsingular part of $\mu$.

\subsection{Large deviations}

In \cite{BrandleChasseigne09-1} and \cite{BrandleChasseigne09-2}, some bounds for
the same nonlocal linear equation were derived, but related to a different problem: estimating the error
when approaching the solution $u$ in all $\R^N$ from solutions $u_R$ of the  Dirichlet problem in
a ball $B_R$, as $R\to\infty$.

The two problems are somehow related in the sense that measuring the difference $(u-u_R)(T)$ amounts to
measuring the total amount of sample paths that can escape the ball $B_R$ between times $0$ and $T$.
And in some sense, this is another way to estimate the nonlocal heat kernel associated to the equation.

Thus, we recover for instance the typical $(\e^{-R\ln R})$-bounds for compactly supported kernels and
$(\e^{-R})$-bounds for exponentially decaying kernels. In the case when $J$ decays at most exponentially,
more bounds are to be found in
\cite{BrandleChasseigne09-1,BrandleChasseigne09-2} which should give the behaviour of the initial trace
for various kernels $J$, according to their decay.

\bigskip

\textbf{Acknowledgements --} The authors would like to warmly thank Emmanuel Lesigne (Tours University, France)
for his help in estimating convolutions of compactly supported kernels.

The authors are supported by project MTM2008-06326-C02-02 (Spain). R. F. is also supported by grant GR58/08-Grupo 920894.


\end{document}